\documentclass[english,twoside,reqno]{amsart}
\usepackage[square,comma]{natbib}
\usepackage[english]{babel} 
\usepackage{exscale}   
\usepackage{amsmath}
\usepackage{amsfonts}
\usepackage{amssymb}
\usepackage{amsxtra}
\usepackage{pstricks}
\usepackage{stmaryrd}
\usepackage{xspace}
\usepackage{epsfig}
\usepackage{mathrsfs}
\usepackage[left=3.0cm,top=3cm,right=3.0cm,bottom=3cm]{geometry}
\numberwithin{equation}{section}
\newtheorem{Theorem}{Theorem}[section]
\newtheorem{Lemma}[Theorem]{Lemma}
\newtheorem{Corollary}[Theorem]{Corollary}
\newtheorem{Proposition}[Theorem]{Proposition}
\newtheorem{Definition}[Theorem]{Definition}
\newtheorem{Remark}[Theorem]{Remark}

\newtheorem{Example}[Theorem]{Example}

\newcommand{\Bset}{\mathbb{B}}
\newcommand{\Cset}{\mathbb{C}}

\newcommand{\Nset}{\mathbb{N}}

\newcommand{\Pset}{\mathbb{P}}

\newcommand{\Sset}{\mathbb{S}}

\newcommand{\cA}{\ensuremath{{\mathcal A}}\xspace}         % Caligr-A
\newcommand{\cB}{\ensuremath{{\mathcal B}}\xspace}         % Caligr-B
\newcommand{\cH}{\ensuremath{{\mathcal H}}\xspace}         % Caligr-H
\newcommand{\cN}{\ensuremath{{\mathcal N}}\xspace}         % Caligr-N
\newcommand{\cT}{\ensuremath{{\mathcal T}}\xspace}         % Caligr-T
\newcommand{\cZ}{\ensuremath{{\mathcal Z}}\xspace}         % Caligr-Z
\newcommand{\ii}{\mathbf{i}}
\newcommand{\jj}{\mathbf{j}}

\newcommand{\1}{\ensuremath{{\rm 1\kern-.25em l}}\xspace}  % identity
 
\newcommand{\Ad}{\mathop{\mathrm{Ad}}}                     % Ad
\newcommand{\trace}{\operatorname{tr}}        % normalized trace
\newcommand{\Aut}[1]{\ensuremath{\operatorname{Aut}(#1)}\xspace}  % Aut
\newcommand{\End}[1]{\ensuremath{\operatorname{End}(#1)}\xspace}  % End
\newcommand{\set}[2]{\mathopen{\{}#1\mathop{|}#2\mathclose{\}}}%set
\newcommand{\vN}{\operatorname{vN}}

\newsavebox{\artin}% begin declaration sigma (20 by 20 grid)
\newcommand{\masterartin}{
\linethickness{1pt}
\qbezier(10,20)(0,10)(-10,0)
\qbezier(-10,20)(-7,17)(-4,14)
\qbezier(10,0)(6,4)(4,6)
}
\newsavebox{\artininv}% begin declaration sigma^-1 (20 by 20 grid)
\newcommand{\masterartininv}{
\linethickness{1pt}
\qbezier(-10,20)(0,10)(10,0)
\qbezier(10,20)(6,16)(4,14) 
\qbezier(-10,0)(-5,5)(-4,6)
}
\newsavebox{\strandr}% begin declaration (20 by 20 grid)
\newcommand{\masterstrandr}{
\linethickness{1pt}
\qbezier(10,20)(10,10)(10,0)
}
\newsavebox{\strandl}% begin declaration (20 by 20 grid)
\newcommand{\masterstrandl}{
\linethickness{1pt}
\qbezier(-10,20)(-10,10)(-10,0)
}
\newsavebox{\horizontaldots}% begin declaration (20 by 20 grid)
\newcommand{\masterhorizontaldots}{
\linethickness{1pt}
\put(5,0){\circle*{2}}
\put(11,0){\circle*{2}}
\put(17,0){\circle*{2}}
}
\begin{document}
\title{Noncommutative independence in the infinite
braid and symmetric group}
\author[R. Gohm]{Rolf Gohm}
\author[C. K\"ostler]{Claus K\"ostler}
\address{Institute of Mathematics and Physics, Aberystwyth University, Aberystwyth, SY23 3BZ,UK}
\address{Institute of Mathematics and Physics, Aberystwyth University, Aberystwyth, SY23 3BZ,UK}
\email{rog@aber.ac.uk}
\email{cck@aber.ac.uk} 
%\subjclass[2000]{Primary ??; Secondary ??}
%\keywords{??} 
%\date{\today}

\begin{abstract}
This is an introductory paper about our recent merge of a noncommutative de
Finetti type result with representations of the infinite braid and symmetric
group which allows to derive factorization properties from symmetries. We
explain some of the main ideas of this approach and work out a constructive
procedure to use in applications. Finally we illustrate the method by
applying it to the theory of group characters.
\end{abstract}

\maketitle

%%%%%%%%%%%%%%%%%%%%%%%%  S E C T I O N %%%%%%%%%%%%%%%%%%%%%%%%%%%%%%%
\section*{Introduction}
\label{section:intro}
%%%%%%%%%%%%%%%%%%%%%%%%%%%%%%%%%%%%%%%%%%%%%%%%%%%%%%%%%%%%%%%%%%%%%%%
In our recent papers \cite{Koes10a,GoKo09a,GoKo10a} a systematical theory is emerging how
representations of the infinite braid and symmetric group on noncommutative
probability spaces give rise to noncommutative conditional independence and
to factorization properties. These papers are rather long and, in parts,
rather technical (and should be consulted for further background and further references), so it is timely to give a short introductory paper. This is done here. 
The diligent reader of the longer versions will notice that
there are also some new turns and twists in this new presentation but our
main objective is to give a readable guide for a main highway through the
forest.

To achieve this we sacrifice generality at several points, for example we
consider single operators as noncommutative random variables instead of the
more general concept of algebra embeddings. Distributional symmetries such
as exchangeability and spreadability are introduced in this setting and
reformulated in terms of endomorphisms. We make the very important
observation that instead of the representation of the infinite symmetric
group $\Sset_\infty$ associated to exchangeability it is enough for the
basic implications to consider a representation of the infinite braid group
$\Bset_\infty$ instead. This observation enormously broadens the range of
applicability when combined with further insights: a constructive procedure
to produce braidable sequences, the equality of the tail algebra of the
random sequence and the fixed point algebra of the representation (in the
minimal situation), last but not least the noncommutative de Finetti theorem
which derives conditional independence over the tail algebra from
spreadability. Taken together we have a systematic way to derive structural
properties, in particular factorization properties of the state, from
symmetries.

To show some of these applications we concentrate on a specific and
beautiful class of examples: the theory of characters of the infinite braid
and symmetric group. We show how our theory can be used to study the left
regular representation of $\Bset_\infty$ and we sketch some of the main
ingredients of a new proof of Thoma's theorem on extremal characters of
$\Sset_\infty$. Open questions that can be studied with these methods,
especially in the braid group setting, suggest themselves.

%%%%%%%%%%%%%%%%%%%%%%%%  S E C T I O N %%%%%%%%%%%%%%%%%%%%%%%%%%%%%%%
\section{Tracial noncommutative probability spaces}
\label{section:prelim}
%%%%%%%%%%%%%%%%%%%%%%%%%%%%%%%%%%%%%%%%%%%%%%%%%%%%%%%%%%%%%%%%%%%%%%
Throughout this paper a (\emph{noncommutative}) \emph{tracial probability space} $(\cA,\trace)$ consists of a von Neumann algebra $\cA$ acting on the separable Hilbert space $\cH$ and a tracial faithful normal state $\trace \colon \cA \to \Cset$. An element $x \in \cA$ is called a (\emph{noncommutative}) \emph{random variable}. If $x_0, x_1, x_2, \ldots$ is a sequence of random variables in $\cA$, we write $\vN(x_0, x_1, x_2, \ldots)$ for the generated von Neumann subalgebra in $\cA$.
%%%%%%%%%%%%%%%%%%%%%%%%%%%%%%%%%%%%%%%%%%%%%%%%%%%%%%%%%%%%%%%%%%%%%%%%%%%%%%%%%
\begin{Example} \normalfont \label{ex:classical}
%%%%%%%%%%%%%%%%%%%%%%%%%%%%%%%%%%%%%%%%%%%%%%%%%%%%%%%%%%%%%%%%%%%%%%%%%%%%%%%%%
This setting covers the classical case of bounded (complex or real valued) random variables $x_i$ on some standard probability space $(\Omega, \Sigma, \Pset)$. More precisely, the $x_i$'s are elements of $L^\infty(\Omega, \Sigma, \Pset)$ and act by (left) multiplication on $L^2(\Omega, \Sigma, \Pset)$, the essentially bounded resp.~square-integrable Lebesgue-measurable functions on $(\Omega, \Sigma, \Pset)$. Further the trace on $L^\infty(\Omega, \Sigma, \Pset)$ is given by the expectation 
$f \mapsto \int_{\Omega} f d\Pset$. 
%%%%%%%%%%%%%%%%%%%%%%%%%%%%%%%%%%%%%%%%%%%%%%%%%%%%%%%%%%%%%%%%%%%%%%%%%%%%%%%%%%
\end{Example}
%%%%%%%%%%%%%%%%%%%%%%%%%%%%%%%%%%%%%%%%%%%%%%%%%%%%%%%%%%%%%%%%%%%%%%%%%%%%%%%%%%

%%%%%%%%%%%%%%%%%%%%%%%%%%%%%%%%%%%%%%%%%%%%%%%%%%%%%%%%%%%%%%%%%%%%%%%%%%%%%%%%%
\begin{Example} \normalfont \label{ex:group}
%%%%%%%%%%%%%%%%%%%%%%%%%%%%%%%%%%%%%%%%%%%%%%%%%%%%%%%%%%%%%%%%%%%%%%%%%%%%%%%%%
An interesting class of noncommutative probability spaces arises in the following manner. Let $G$ be a countable group. A positive definite function $\chi\colon G \to \Cset$ is called a \emph{character} if $\chi$ is constant on conjugacy classes of $G$ and  normalized at the identity $e$ of $G$. It is a folklore result (reviewed in \cite{GoKo10a}, for example) that a character $\chi$ gives rise to a unitary representation $\pi$ of $G$ on a separable Hilbert space $\cH$ and a vector $\xi \in \cH$ such that 
\[
\chi(g^{-1}h) = \langle \pi(g) \xi, \pi(h)\xi \rangle   
\]
for all $g, h \in G$. Further $\trace := \langle \xi, \bullet \,\xi \rangle$ is a tracial faithful normal state on $\cA = \vN\set{\pi(g)}{g\in G}$. Conversely, restricting a tracial state on a group algebra gives rise to a character of the group. It is well-known that $\cA$ is a factor (in the sense of von Neumann algebras) if and only if $\chi$ is extremal. Note also that each unitary $\pi(g)$ is a noncommutative random variable.   
%%%%%%%%%%%%%%%%%%%%%%%%%%%%%%%%%%%%%%%%%%%%%%%%%%%%%%%%%%%%%%%%%%%%%%%%%%%%%%%%%%
\end{Example}
%%%%%%%%%%%%%%%%%%%%%%%%%%%%%%%%%%%%%%%%%%%%%%%%%%%%%%%%%%%%%%%%%%%%%%%%%%%%%%%%%%
The following two choices for the group $G$ in Example \ref{ex:group} are of particular interest in the sequel.
%%%%%%%%%%%%%%%%%%%%%%%%%%%%%%%%%%%%%%%%%%%%%%%%%%%%%%%%%%%%%%%%%%%%%%%%%%%%%%%%%
\begin{Example} \normalfont \label{ex:braid}
%%%%%%%%%%%%%%%%%%%%%%%%%%%%%%%%%%%%%%%%%%%%%%%%%%%%%%%%%%%%%%%%%%%%%%%%%%%%%%%%%
The infinite braid group $\Bset_\infty$ is the inductive limit of 
the braid groups $\Bset_n$ for $n \to \infty$. It is presented by the 
Artin generators $\sigma_1,\sigma_2, \ldots$ satisfying the relations
\begin{align}
&&\sigma_i \sigma_{j} \sigma_i 
&= \sigma_{j} \sigma_i \sigma_{j} 
&\text{if $ \; \mid i-j \mid\, = 1 $;}&& \tag{B1} \label{eq:B1}\\
&&\sigma_i \sigma_j 
&= \sigma_j \sigma_i  
&\text{if $ \; \mid i-j \mid\, > 1 $.}&& \tag{B2} \label{eq:B2}
\end{align} 
The Artin generator $\sigma_i$ and its inverse $\sigma_i^{-1}$ are
presented as geometric braids according to Figure \ref{figure:artin}.
%%%%%%%%%%%%%%%%%%%%%%%%%%%%%%%%%%%%%%%%%%%%%%%%%%%%%%%%%%%%%%%%%%%%%%%%%%%
%%%%%%%%%%%%%%%%%   FIGURE 0a -- BEGIN  %%%%%%%%%%%%%%%%%%%%%%%%%%%%%%%%%%%
%%%%%%%%%%%%%%%%%%%%%%%%%%%%%%%%%%%%%%%%%%%%%%%%%%%%%%%%%%%%%%%%%%%%%%%%%%%
\begin{figure}[h]
\setlength{\unitlength}{0.3mm}
\begin{picture}(360,35)%(0,0) optional shift of picture
%\put(0,0){\line(0,1){40}} %reference coordinate system
%\put(0,0){\line(1,0){340}} %reference coordinate system 
%%%%%%%%%%%%%%%%%%%%%%%%%%%%%%%%%%%%%%%%%%%%%%%%%%%%%%%%%%%%
%%%%%%%%%%% Definition of partial pictures - BEGIN %%%%%%%%%
%%%%%%%%%%%%%%%%%%%%%%%%%%%%%%%%%%%%%%%%%%%%%%%%%%%%%%%%%%%%
\savebox{\artin}(20,20)[1]{\masterartin} 
\savebox{\artininv}(20,20)[1]{\masterartininv} 
\savebox{\strandr}(20,20)[1]{\masterstrandr} 
\savebox{\strandl}(20,20)[1]{\masterstrandl} 
\savebox{\horizontaldots}(20,20)[1]{\masterhorizontaldots}
%%%%%%%%%%%%%%%%%%%%%%%%%%%%%%%%%%%%%%%%%%%%%%%%%%%%%%%%%%%%
%%%%%%%%%%% Definition of partial pictures - END   %%%%%%%%%
%%%%%%%%%%%%%%%%%%%%%%%%%%%%%%%%%%%%%%%%%%%%%%%%%%%%%%%%%%%%
%%%%%%%%% the grid is 20 by 20 for partial pictures %%%%%%%%
%%%%%%%%%%%%%%%%%%%%%%%%%%%%%%%%%%%%%%%%%%%%%%%%%%%%%%%%%%%%
\put(0,00){\usebox{\strandl}}  
\put(0,00){\usebox{\strandr}}
\put(20,00){\usebox{\horizontaldots}}
\put(60,0){\usebox{\strandl}}
\put(80,0){\usebox{\artin}}
\put(100,0){\usebox{\strandr}}
\put(120,0){\usebox{\horizontaldots}}
%%%%%%%%%%%%%%%%%%%%%%%%%%%%%%%%
\put(-1,25){\footnotesize{$0$}}
\put(19,25){\footnotesize{$1$}}
\put(72,25){\footnotesize{$i-1$}}
\put(100,25){\footnotesize{$i$}}
%%%%%%%%%%%%%%%%%%%%%%%%%%%%%%%%
%%%%%%%%%%%%%%%%%%%%%%%%%%%%%%%%
\put(200,00){\usebox{\strandl}}  
\put(200,00){\usebox{\strandr}}
\put(220,00){\usebox{\horizontaldots}}
\put(260,00){\usebox{\strandl}}
\put(280,00){\usebox{\artininv}}
\put(300,00){\usebox{\strandr}}
\put(320,00){\usebox{\horizontaldots}}
%%%%%%%%%%%%%%%%%%%%%%%%%%%%%%%%
\put(199,25){\footnotesize{$0$}}
\put(219,25){\footnotesize{$1$}}
\put(272,25){\footnotesize{$i-1$}}
\put(300,25){\footnotesize{$i$}}
%%%%%%%%%%%%%%%%%%%%%%%%%%%%%%%%
\end{picture}
\caption{Artin generators $\sigma_i$ (left) and $\sigma_i^{-1}$ (right)}
\label{figure:artin}
\end{figure}
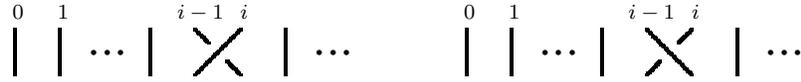
%%%%%%%%%%%%%%%%%%%%%%%%%%%%%%%%%%%%%%%%%%%%%%%%%%%%%%%%%%%%%%%%%%%%%%%%%%%
%%%%%%%%%%%%%%%%%   FIGURE 0a -- END   %%%%%%%%%%%%%%%%%%%%%%%%%%%%%%%%%%%%
%%%%%%%%%%%%%%%%%%%%%%%%%%%%%%%%%%%%%%%%%%%%%%%%%%%%%%%%%%%%%%%%%%%%%%%%%%%
%%%%%%%%%%%%%%%%%%%%%%%%%%%%%%%%%%%%%%%%%%%%%%%%%%%%%%%%%%%%%%%%%%%%%%%%%%%%%%%%%%
We will later meet a special choice of the character, 
$\chi(g)= 0$ for $g \neq e$. This choice yields for the probability space $\cA$ the group von Neumann algebra $L(\Bset_\infty)$, a non-hyperfinite II$_1$ factor.  
\end{Example}
%%%%%%%%%%%%%%%%%%%%%%%%%%%%%%%%%%%%%%%%%%%%%%%%%%%%%%%%%%%%%%%%%%%%%%%%%%%%%%%%%% 
%%%%%%%%%%%%%%%%%%%%%%%%%%%%%%%%%%%%%%%%%%%%%%%%%%%%%%%%%%%%%%%%%%%%%%%%%%%%%%%%%
\begin{Example} \normalfont
%%%%%%%%%%%%%%%%%%%%%%%%%%%%%%%%%%%%%%%%%%%%%%%%%%%%%%%%%%%%%%%%%%%%%%%%%%%%%%%%%
The infinite symmetric group $\Sset_\infty$ is the inductive limit of 
the symmetric groups $\Sset_n$ for $n \to \infty$, in other words: the group 
of all finite permutations. It is presented by the Coxeter generators $\sigma_1,\sigma_2, \ldots$ satisfying \eqref{eq:B1}, \eqref{eq:B2} and the additional relations
\begin{align}
&&\sigma_i^2  
&= e  
&\text{($i \in \Nset$).}&& \tag{S} \label{eq:S}
\end{align} 
Here the Coxeter generators are realized as permutations acting on the set $\Nset_0= \{0,1,2,\ldots\}$ such that $\sigma_i$ sends $(0,1,\ldots,
i-1,i, \ldots)$ to $(0,1,\ldots, i,i-1, \ldots)$. Extremal characters (except the trivial or alternating character) yield a hyperfinite II$_{1}$ factor as probability space $\cA$. Note also that the random variables $\pi(\sigma_i)$ are selfadjoint and 
correspond to classical $\{-1,1\}$-valued random variables embedded into $\cA$ in a noncommutative fashion according to the braid relations \eqref{eq:B1} and \eqref{eq:B2}.  
%%%%%%%%%%%%%%%%%%%%%%%%%%%%%%%%%%%%%%%%%%%%%%%%%%%%%%%%%%%%%%%%%%%%%%%%%%%%%%%%%%
\end{Example}
%%%%%%%%%%%%%%%%%%%%%%%%%%%%%%%%%%%%%%%%%%%%%%%%%%%%%%%%%%%%%%%%%%%%%%%%%%%%%%%%%%

We fix some additional notation as needed in the sequel. If $\cB$ is a von Neumann subalgebra of $\cA$, then $E_\cB$ denotes the $\trace$-preserving conditional expectation from $\cA$ onto $\cB$. (Such a conditional expectation uniquely exists in our tracial setting. The non-tracial case would require an additional modular condition.) $\End{\cA,\trace}$ denotes the $\trace$-preserving endomorphisms of the von Neumann algebra $\cA$. The fixed point algebra of $\alpha \in \End{\cA,\trace}$ is denoted by $\cA^\alpha$. Similarly, $\Aut{\cA, \trace}$ denotes the $\trace$-preserving automorphisms of $\cA$. 

%%%%%%%%%%%%%%%%%%%%%%%%%%%%%%%%%%%%%%%%%%%%%%%%%%%%%%%%%%%%%%%%%%%%%%%%%%%%%%%%%%
\section{Distributional symmetries}\label{section:ds}
%%%%%%%%%%%%%%%%%%%%%%%%%%%%%%%%%%%%%%%%%%%%%%%%%%%%%%%%%%%%%%%%%%%%%%%%%%%%%%%%%%
We are interested in certain distributional symmetries of infinite sequences of random variables. Given the probability space $(\cA,\trace)$, then two sequences $(x_n)_{n \ge 0}$ and $(y_n)_{n\ge0}$ in $\cA$ are said to have the same joint *-moments, in symbols:
\[
(x_0, x_1, x_2, \ldots) \stackrel{\trace}{=}  (y_0, y_1, y_2, \ldots),
\]
if, for every $n \in \Nset$,  
\begin{align}\label{eq:def-ds}
\trace(x_{\ii(1)}^{\epsilon_1} x_{\ii(2)}^{\epsilon_2} \cdots x_{\ii(n)}^{\epsilon_n})
=\trace(y_{\ii(1)}^{\epsilon_1} y_{\ii(2)}^{\epsilon_2} \cdots y_{\ii(n)}^{\epsilon_n})
\end{align}   
for all $\ii \colon \{1,2,\ldots, n\} \to \Nset_0$
and $\epsilon_1, \epsilon_2, \ldots, \epsilon_n \in \{1, *\}$. 
Note that this family of equations extends immediately to polynomials 
in the random variables such that they have the same joint *-distribution:
\[
\trace\Big(P(x_1, x_1^*,\ldots, x_s, x_s^*)\Big) 
=  \trace\Big(P(y_1, y_1^*,\ldots, y_s, y_s^*)\Big)
\qquad 
(s \in \Nset, P \in \Cset\langle x_1^{}, x_1^*, \ldots, x_s^{}, x_s^* \rangle).
\] 
We refer the reader to \cite{NiSp06a} for more detailed information on joint *-monomials and joint *-distributions.

%%%%%%%%%%%%%%%%%%%%%%%%%%%%%%%%%%%%%%%%
\begin{Definition}\label{def:ds} \normalfont
%%%%%%%%%%%%%%%%%%%%%%%%%%%%%%%%%%%%%%5
Let $(\cA,\trace)$ be a tracial probability space. The sequence $(x_n)_{n \ge 0} \subset\cA$ is 
\begin{enumerate}
\item[(i)]
\emph{exchangeable} if
\[
(x_0, x_1, x_2, \ldots) \stackrel{\trace}{=}
(x_{\pi(0)}, x_{\pi(1)}, x_{\pi(2)}, \ldots)
\] 
for all $\pi \in \Sset_\infty$;  
\item[(ii)]
\emph{spreadable} if
\[
(x_0, x_1, x_2, \ldots) \stackrel{\trace}{=}
(x_{n_0}, x_{n_1}, x_{n_2}, \ldots)
\] 
for any (increasing) subsequence $(n_0, n_1, n_2, \ldots)$ of $(0,1,2,\ldots)$;  
\item[(iii)]
\emph{stationary} if
\[
(x_0, x_1, x_2, \ldots) \stackrel{\trace}{=}
(x_{k}, x_{k+1}, x_{k+2}, \ldots)
\] 
for all $k \in \Nset$; 
\item[(iv)]
\emph{identically distributed} if 
\[
(x_0, x_0, x_0, \ldots) \stackrel{\trace}{=}
(x_{k}, x_{k}, x_{k}, \ldots)
\] 
for all $k \in \Nset$. 
\end{enumerate}
\end{Definition}
%%%%%%%%%%%%%%%%%%%%%%%%%%%%%%%%%%%%%%%%%%%%%%%%%%%%%%%%%%%%%%%%%%%%%%%%%%%%%%%
\begin{Lemma} \label{lem:ds}
%%%%%%%%%%%%%%%%%%%%%%%%%%%%%%%%%%%%%%%%%%%%%%%%%%%%%%%%%%%%%%%%%%%%%%%%%%%%%%%
One has the following hierarchy of distributional symmetries in Definition \ref{def:ds}:
\begin{center}
 (i) $\Rightarrow$ (ii) $\Rightarrow$ (iii) $\Rightarrow$ (iv)
\end{center} 
%%%%%%%%%%%%%%%%%%%%%%%%%%%%%%%%%%%%%%%%%%%%%%%%%%%%%%%%%%%%%%%%%%%%%%%%%%%%%%%
\end{Lemma}
%%%%%%%%%%%%%%%%%%%%%%%%%%%%%%%%%%%%%%%%%%%%%%%%%%%%%%%%%%%%%%%%%%%%%%%%%%%%%%%
\begin{proof}
Exchangeability implies spreadability since each formula for the joined *-moments involves only finitely many $x_i$'s. 
If $n_0 < n_1 < n_2 < ... < n_k$ is the beginning of a subsequence
of $(0,1,2,\ldots)$, then we can always find a permutation $\pi \in \Sset_\infty$ such that $\pi(i) = n_i$. This shows (i) $\Rightarrow$ (ii). All other implications are evident. 
\end{proof}
%%%%%%%%%%%%%%%%%%%%%%%%%%%%%%%%%%%%%%%%%%%%%%%%%%%%%%%%%%%%%%%%%%%%%%%%%%%%%%%%%%%%%%%%%%%%%%%%%%%%%%%%%

Equivalent formulations of (i) to (iii) in Definition \ref{def:ds} are available in a minimal setting of the probability space. Here we meet again the infinite symmetric group, but now acting as automorphisms on the von Neumann algebra. 
%%%%%%%%%%%%%%%%%%%%%%%%%%%%%%%%%%%%%%%%%%%%%%%%%%%%%%%%%%%%%%%%%%%%%%
\begin{Proposition}\label{prop:ds}
%%%%%%%%%%%%%%%%%%%%%%%%%%%%%%%%%%%%%%%%%%%%%%%%%%%%%%%%%%%%%%%%%%%%%%
Let $(\cA, \trace)$ be a tracial probability space and suppose $\cA = \vN(x_0, x_1, x_2, \ldots)$. 
\begin{enumerate}
\item 
$(x_n)_{n \ge 0}$ is exchangeable iff there exists a representation $\rho \colon \Sset_\infty \to \Aut{\cA,\trace}$ such that
\begin{align}
x_n &= \rho(\sigma_n \sigma_{n-1} \cdots \sigma_1) x_0  & (n \ge 1)& \tag{PR} \label{eq:PR}\\
x_0 &= \rho(\sigma_n)x_0  &(n \ge 2)&. \tag{L} \label{eq:L} 
\end{align}
\item 
$(x_n)_{n \ge 0}$ is spreadable iff there exist endomorphisms $(\alpha_N)_{N \ge 0} \subset \End{\cA,\trace}$, also called \emph{partial shifts}, such that
\begin{eqnarray} \label{eq:partialshift}
\alpha_{N}(x_n) = 
\begin{cases}
x_n & \text{ for $n < N$}\\
x_{n+1} & \text{ for $n \ge N$}
\end{cases}
\qquad \qquad (n, N \in \Nset_0).
\end{eqnarray}
\item 
$(x_n)_{n \ge 0}$ is stationary iff there exists an endomorphism $\alpha \subset \End{\cA,\trace}$ such that
\begin{eqnarray*}
\alpha(x_n) = x_{n+1}. \qquad \qquad (n \in \Nset_0).
\end{eqnarray*}
\end{enumerate}
%%%%%%%%%%%%%%%%%%%%%%%%%%%%%%%%%%%%%%%%%%%%%%%%%%%%%%%%%%%%%%%%%%%%%%
\end{Proposition}
%%%%%%%%%%%%%%%%%%%%%%%%%%%%%%%%%%%%%%%%%%%%%%%%%%%%%%%%%%%%%%%%%%%%%%
The proof is not difficult. If required details may be found for (i)
in \cite[Theorem 1.9]{GoKo09a}, for (ii) in \cite[Lemma 8.6]{Koes10a}
and for (iii) in \cite[Lemma 2.5]{Koes10a}. Note that the partial shift 
$\alpha_0$ is the shift $\alpha$ in the stationary setting. 

The equivalent characterization in (i) provides us with a \emph{constructive procedure} to obtain exchangeable sequences. Suppose we have a tracial probability space $(\cA,\trace)$ which is equipped with a representation of the infinite symmetric group, $\rho \colon \Sset_\infty \to \Aut{\cA,\trace}$. Now choose an element $x_0 \in \bigcap_{n \ge 2}\cA^{\rho(\sigma_n)}$, hence $x_0$ satisfies \eqref{eq:L}. Then an exchangeable sequence $(x_n)_{n \ge 0}$ can be constructed by using \eqref{eq:PR}. 

%%%%%%%%%%%%%%%%%%%%%%%%%%%%%%%%%%%%%%%%%%%%%%%%%%%%%%%%%%%%%%%%%%%%%%%%%%%%%%%%%
\begin{Remark}\normalfont
%%%%%%%%%%%%%%%%%%%%%%%%%%%%%%%%%%%%%%%%%%%%%%%%%%%%%%%%%%%%%%%%%%%%%%%%%%%%%%%%%
This constructive procedure may yield an exchangeable sequence which does not generate $\cA$ (this is evident from the choice $x_0= \1$). But we can always return to a minimal setting of the tracial probability space, by restricting the trace $\trace$ to $\vN\{x_0,x_1, \ldots\}$. Note that the sequence $(x_n)_{n \ge 0}$ is also exchangeable after this restriction and thus also the representation $\rho$ restricts to $\vN\{x_0,x_1, \ldots\}$. Thus \emph{any} exchangeable sequence
can be obtained from the constructive procedure.
%%%%%%%%%%%%%%%%%%%%%%%%%%%%%%%%%%%%%%%%%%%%%%%%%%%%%%%%%%%%%%%%%%%%%%%%%%%%%%%%%% 
\end{Remark}    
%%%%%%%%%%%%%%%%%%%%%%%%%%%%%%%%%%%%%%%%%%%%%%%%%%%%%%%%%%%%%%%%%%%%%%%%%%%%%%%%%%
%%%%%%%%%%%%%%%%%%%%%%%%%%%%%%%%%%%%%%%%%%%%%%%%%%%%%%%%%%%%%%%%%%%%%%%%%%%%%%%%%
%\begin{Remark}\normalfont
%%%%%%%%%%%%%%%%%%%%%%%%%%%%%%%%%%%%%%%%%%%%%%%%%%%%%%%%%%%%%%%%%%%%%%%%%%%%%%%%%%
Presently no constructive procedure is known to produce \emph{all} spreadable sequences.
We will see below that the constructive procedure for exchangeable sequences extends to a `braided' setting where the role of $\Sset_\infty$ is taken by the infinite braid group $\Bset_\infty$. In particular this provides us with an interesting class of spreadable sequences which may not be exchangeable.
%%%%%%%%%%%%%%%%%%%%%%%%%%%%%%%%%%%%%%%%%%%%%%%%%%%%%%%%%%%%%%%%%%%%%%%%%%%%%%%%%% 
%\end{Remark}    
%%%%%%%%%%%%%%%%%%%%%%%%%%%%%%%%%%%%%%%%%%%%%%%%%%%%%%%%%%%%%%%%%%%%%%%%%%%%%%%%%%
Next we motivate this `braided' extension by providing an alternative proof of the
simple fact that exchangeability implies spreadability, based on the equivalent
characterizations in Proposition \ref{prop:ds}.  
%%%%%%%%%%%%%%%%%%%%%%%%%%%%%%%%%%%%%%%%%%%%%%%%%%%%%%%%%%%%%%%%%%%%%%%%%%%%%%%%%%%%%%%%%%%
\begin{Lemma}\label{lem:ex2sp}
%%%%%%%%%%%%%%%%%%%%%%%%%%%%%%%%%%%%%%%%%%%%%%%%%%%%%%%%%%%%%%%%%%%%%%%%%%%%%%%%%%%%%%%%%%%
Suppose $(\cA, \trace)$ is equipped with the representation $\rho\colon \Sset_\infty \to \Aut{\cA,\trace}$. Let $x_0 \in \bigcap_{n \ge 2} \cA^{\rho(\sigma_n)}$ and $(x_n)_{n \ge 0}$ the exchangeable sequence  obtained from the constructive procedure. Then, for each $N \in \Nset$, the map
\[
x_n \mapsto \alpha_N(x_n):= \rho(\sigma_{N+1}\sigma_{N+2} \cdots \sigma_{N+k})(x_n),  
\] 
with $k$ sufficiently large, satisfies \eqref{eq:partialshift} and extends to an 
endomorphism in $(\widetilde{\cA}, \widetilde{\trace})$, where $\widetilde{\cA}= \vN(x_0, x_1, x_2, \ldots)$ and $\widetilde{\trace}= \trace|_{\widetilde{\cA}}$.
%%%%%%%%%%%%%%%%%%%%%%%%%%%%%%%%%%%%%%%%%%%%%%%%%%%%%%%%%%%%%%%%%%%%%%%%%%%%%%%%%%%%%%%%%%%%
\end{Lemma}
%%%%%%%%%%%%%%%%%%%%%%%%%%%%%%%%%%%%%%%%%%%%%%%%%%%%%%%%%%%%%%%%%%%%%%%%%%%%%%%%%%%%%%%%%%%%
\begin{proof}
By exchangeability we have $x_n= \rho(\sigma_{n} \sigma_{n-1} \cdots \sigma_1) (x_0)$ for 
$n \ge 1$. Let us start with the case $n <N$. Since $N+1-n>1$,
\begin{eqnarray*}
\alpha_N(x_n)&=&  \rho\Big(
(\sigma_{N+1}\sigma_{N+2} \cdots \sigma_{N+k})
(\sigma_{n} \sigma_{n-1} \cdots \sigma_1)\Big) 
(x_0)\\
&\stackrel{\eqref{eq:B2}}{=}& 
\rho\Big((\sigma_{n} \sigma_{n-1} \cdots \sigma_1)
(\sigma_{N+1}\sigma_{N+2} \cdots \sigma_{N+k})\Big)
(x_0)\\
&\stackrel{\eqref{eq:L}}{=}& x_n.
\end{eqnarray*} 
Now consider the case $n\ge N$. Due to the previous arguments it
suffices to consider $k = n+1-N$. Thus we have 
\begin{eqnarray} \label{eqn:exch2spread}
\alpha_N(x_n)&=&  \rho\Big(
(\sigma_{N+1}\sigma_{N+2} \cdots \sigma_{n}\sigma_{n+1})
(\sigma_{n} \sigma_{n-1} \cdots \sigma_1)\Big) 
(x_0).
\end{eqnarray}   
The braid relations \eqref{eq:B1} and \eqref{eq:B2} supply us with the
intertwining property
\[
\sigma_{l} (\sigma_{n+1}\sigma_{n}\cdots \sigma_{1})
=  (\sigma_{n+1}\sigma_{n}\cdots \sigma_{1}) \sigma_{l+1} 
\]
for $N \le l \le n$. Its repeated application in \eqref{eqn:exch2spread} moves $\sigma_{N+1} \sigma_{N+2} \cdots \sigma_n$ to the right hand side of $\sigma_{n+1}(\sigma_n \sigma_{n-1}\cdots \sigma_1)$ such that each factor vanishes due to \eqref{eq:L}. Altogether we have shown that
\begin{eqnarray*}
\alpha_{N}(x_n) = 
\begin{cases}
x_n & \text{ for $n < N$}\\
x_{n+1} & \text{ for $n \ge N$}.
\end{cases} 
\end{eqnarray*} 
The monomials $x_{i_1} \cdots x_{i_m}$ form a weak*-total set in $\widetilde{\cA}$. Since $\alpha_N$ is implemented on $x_{i_1} \cdots x_{i_m}$ by an automorphism of $\cA$ satisfying
$\trace \circ \alpha_N = \trace$, the $\Cset$-linear multiplicative extension of $\alpha_N$ defines an endomorphism in $\End{\widetilde{\cA},\widetilde{\trace}}$.  \end{proof}   
%%%%%%%%%%%%%%%%%%%%%%%%%%%%%%%%%%%%%%%%%%%%%%%%%%%%%%%%%%%%%%%%%%%%%%%%%%%%%%%%

A review of above proof shows that all algebraic arguments rely on the braid relations \eqref{eq:B1}, \eqref{eq:B2}, the localization \eqref{eq:L} and the
product representation \eqref{eq:PR}. We did not use the relations
\eqref{eq:S} of the infinite symmetric group. This motivates the following new
symmetry for sequences in noncommutative probability spaces.

%%%%%%%%%%%%%%%%%%%%%%%%%%%%%%%%%%%%%%%%%%%%%%%%%%%%%%%%%%%%%%%%%%%%
\begin{Definition} \normalfont
%%%%%%%%%%%%%%%%%%%%%%%%%%%%%%%%%%%%%%%%%%%%%%%%%%%%%%%%%%%%%%%%%%%%
Let $(\cA,\trace)$ be a tracial probability space. The random variables $(x_n)_{n \ge 0} \subset \cA$ are said to be \emph{braidable} if there exists a representation $\rho\colon \Bset_\infty \to \Aut{\cA,\trace}$ such that 
\begin{align}
x_n &= \rho(\sigma_n \sigma_{n-1} \cdots \sigma_1) (x_0)  & (n \ge 1),& \notag\\
x_0 &= \rho(\sigma_n)(x_0)  &(n \ge 2)&. 
\end{align}
Here $\sigma_i$ denotes the Artin generator braiding the $(i-1)$-th and
$i$-th strand. 
%%%%%%%%%%%%%%%%%%%%%%%%%%%%%%%%%%%%%%%%%%%%%%%%%%%%%%%%%%%%%%%%%%%%
\end{Definition}
%%%%%%%%%%%%%%%%%%%%%%%%%%%%%%%%%%%%%%%%%%%%%%%%%%%%%%%%%%%%%%%%%%%%
As in the case of exchangeability, the constructive procedure applies
again to obtain braidable sequences. So given the braid group 
representation $\rho \colon \Bset_\infty \to \Aut{\cA,\trace}$, choose
an element $x_0 \in \bigcap_{n \ge 2} \cA^{\rho(\sigma_n)}$ and use 
\eqref{eq:PR} to obtain all $x_n$'s. Most importantly, replacing 
$\Sset_\infty$ by $\Bset_\infty$, the proof of Lemma \ref{lem:ex2sp} 
directly transfers to the `braided' setting:

%%%%%%%%%%%%%%%%%%%%%%%%%%%%%%%%%%%%%%%%%%%%%%%%%%%%%%%%%%%%%%%%%%%%%
\begin{Theorem}\label{thm:braidable} 
%%%%%%%%%%%%%%%%%%%%%%%%%%%%%%%%%%%%%%%%%%%%%%%%%%%%%%%%%%%%%%%%%%%%%
A braidable sequence $(x_n)_{n \ge 0}$ is spreadable.
%%%%%%%%%%%%%%%%%%%%%%%%%%%%%%%%%%%%%%%%%%%%%%%%%%%%%%%%%%%%%%%%%%%%%
\end{Theorem}
%%%%%%%%%%%%%%%%%%%%%%%%%%%%%%%%%%%%%%%%%%%%%%%%%%%%%%%%%%%%%%%%%%%%%
It is easy to see that exchangeability implies braidability. Thus we 
can insert `braidability' in the hierarchy of distributional symmetries
of Lemma \ref{lem:ds} between `exchangeability' and `spreadability'.

%%%%%%%%%%%%%%%%%%%%%%%%%%%%%%%%%%%%%%%%%%%%%%%%%%%%%%%%%%%%%%%%%%%%%
\begin{Remark} \normalfont
%%%%%%%%%%%%%%%%%%%%%%%%%%%%%%%%%%%%%%%%%%%%%%%%%%%%%%%%%%%%%%%%%%%%%
We should warn the reader that, in contrast to exchangeability, 
a braidable sequence $(x_n)_{n \ge 0}$ may go along with a braid group representation
$\rho$ which does not restrict to  $\vN(x_0, x_1, x_2, \ldots)$
in the non-minimal case.  
 %%%%%%%%%%%%%%%%%%%%%%%%%%%%%%%%%%%%%%%%%%%%%%%%%%%%%%%%%%%%%%%%%%%%%
\end{Remark} 
%%%%%%%%%%%%%%%%%%%%%%%%%%%%%%%%%%%%%%%%%%%%%%%%%%%%%%%%%%%%%%%%%%%%%
As a by-product of the constructive procedure we obtain the following
fixed point characterizations for braidable sequences. We remind that 
the tail algebra of a sequence $(x_n)_{n \ge 0} \subset \cA$
is given by 
\[
\cA^\cT := \bigcap_{n \ge 0 } \vN(x_n, x_{n+1}, x_{n+2}, \ldots).
\]
%%%%%%%%%%%%%%%%%%%%%%%%%%%%%%%%%%%%%%%%%%%%%%%%%%%%%%%%%%%%%%%%%%%%%%
\begin{Theorem} \label{thm:fix}
%%%%%%%%%%%%%%%%%%%%%%%%%%%%%%%%%%%%%%%%%%%%%%%%%%%%%%%%%%%%%%%%%%%%%%
Let $(\cA,\trace)$ be a tracial probability space and suppose that
the braidable sequence $(x_n)_{n \ge 0}$ generates the von Neumann algebra $\cA$. 
Then we have 
\[
\cA^\cT = \cA^{\rho(\Bset_\infty)}= \cA^{\alpha}. 
\]
Here $\cA^\cT$ is the tail algebra of $(x_n)_{n \ge 0}$ and 
$\cA^{\rho(\Bset_\infty)}$ is the fixed point algebra of the 
representation $\rho$ which implements braidability. Finally, $\cA^\alpha$ 
is the fixed point algebra of the shift $\alpha$ satisfying 
$\alpha(x_n) = x_{n+1}$ for all $n \in \Nset_0$. 
%%%%%%%%%%%%%%%%%%%%%%%%%%%%%%%%%%%%%%%%%%%%%%%%%%%%%%%%%%%%%%%%%%%%%%
\end{Theorem}
%%%%%%%%%%%%%%%%%%%%%%%%%%%%%%%%%%%%%%%%%%%%%%%%%%%%%%%%%%%%%%%%%%%%%%
\begin{proof}
We will show that $\cA^{\cT} \subset \cA^{\rho(\Bset_\infty)} \subset \cA^{\alpha} \subset \cA^{\cT}$. 

We start with the first inclusion. Let $x \in \cA^\cT$ be fixed. It suffices to 
show that $\rho(\sigma_k)(x)= x$ for any $k \in \Nset$. After choosing the number $k$ we approximate $x \in \cA^\cT$ by elements of 
\[
\bigcup_{s \ge 0}\Cset\langle x_n, x_n^*, x_{n+1}, x_{n+1}^*, \ldots, x_{n+s}, x_{n+s}^*\rangle
\] 
for any fixed $n$. For $k < i$,
\begin{eqnarray*}
\rho(\sigma_k)(x_i) 
&=& \rho\big(\sigma_k (\sigma_i \cdots \sigma_1 )\big)(x_0)\\
&=& \rho\big( (\sigma_i \cdots \sigma_1 )\sigma_{k+1}\big)(x_0)\\
&=& \rho\big( \sigma_i \cdots \sigma_1 \big)(x_0)\\
&=& x_i
\end{eqnarray*}
from \eqref{eq:B1} and \eqref{eq:B2}. This entails $\cA^\cT \subset \cA^\rho(\Bset_\infty)$.   

The second inclusion $\cA^{\rho(\Bset_\infty)} \subset \cA^{\alpha}$ follows from
$\alpha(x_n) = \rho(\sigma_1 \sigma_2 \ldots \sigma_\ell)(x_n)$ (for $\ell$ sufficiently large, compare Lemma \ref{lem:ex2sp}) and a standard argument on the approximation
of $x\in \cA$ by elements of the *-algebra $\bigcup_{s \ge 0}\Cset\langle x_0, x_0^*, x_1, x_1^*, \ldots, x_s, x_s^*\rangle$, which is weak*-dense in $\cA$.     
  
The last inclusion is immediate from 
\[
\cA^{\alpha} \subset \bigcap_{n \ge 0}\alpha^n(\cA)
= \bigcap_{n \ge 0} \alpha^n \vN(x_0, x_1, \ldots)
= \bigcap_{n \ge 0} \vN(x_n, x_{n+1}, \ldots) = \cA^{\cT},
\] 
where we have used again that $\cA$ is generated by the $x_i$'s.  
\end{proof}
%%%%%%%%%%%%%%%%%%%%%%%%%%%%%%%%%%%%%%%%%%%%%%%%%%%%%%%%%%%%%%%%%%%%%
Roughly speaking, the previous theorem allows us to upgrade joint *-distributions
to operator-valued joint *-distributions, by replacing the trace $\trace$ by 
the conditional expectation $E_\cT$ onto the tail algebra of a stationary 
sequence $(x_n)_{n \ge 0}$ which generates $\cA$. To be more precise, let $(\cA, \trace)$ 
be a tracial probability space and $(x_n)_{n \ge 0}, (y_n)_{n \ge 0} \subset \cA$ two stationary 
sequences which have the same tail algebra $\cA^\cT$ and each of them generates $\cA$. 
Then we write
\[
(x_0, x_1, x_2, \ldots) \stackrel{E_\cT}{=}  (y_0, y_1, y_2, \ldots),
\]
if, for every $n \in \Nset$,  
\begin{align}\label{eq:def-ds-tail}
E_\cT(x_{\ii(1)}^{\epsilon_1} x_{\ii(2)}^{\epsilon_2} \cdots x_{\ii(n)}^{\epsilon_n})
=E_\cT(y_{\ii(1)}^{\epsilon_1} y_{\ii(2)}^{\epsilon_2} \cdots y_{\ii(n)}^{\epsilon_n})
\end{align}   
for all $\ii \colon \{1,2,\ldots, n\} \to \Nset_0$
and $\epsilon_1, \epsilon_2, \ldots \epsilon_n \in \{1, *\}$.

Using Theorem \ref{thm:fix} it follows from standard arguments as in \cite[Lemma 7.6]{Koes10a} the following `lifted' version of distributional symmetries.
%%%%%%%%%%%%%%%%%%%%%%%%%%%%%%%%%%%%%%%%%%%%%%%%%%%%%%%%%%%%%%%%%%%%%%
\begin{Corollary}\label{cor:tail-spread}
%%%%%%%%%%%%%%%%%%%%%%%%%%%%%%%%%%%%%%%%%%%%%%%%%%%%%%%%%%%%%%%%%%%%%%
Let $(\cA,\trace)$ be a tracial probability space and suppose the 
stationary sequence $(x_n)_{n \ge 0}$ generates $\cA$.
\begin{enumerate}
\item 
If $(x_n)_{n \ge 0}$ is exchangeable then 
$
(x_0, x_1, x_2, \ldots) \stackrel{E_\cT}{=}  (x_{\pi(0)}, x_{\pi(1)}, x_{\pi(2)}, \ldots),
$
for all $\pi \in \Sset_\infty$.
\item
If $(x_n)_{n \ge 0}$ is spreadable then 
$
(x_0, x_1, x_2, \ldots) \stackrel{E_\cT}{=}  (x_{n_0}, x_{n_1}, x_{n_2}, \ldots),
$
for all (increasing) subsequences $(n_0, n_1, n_2, \ldots)$ of $(0,1,2,\ldots)$.
\end{enumerate}
%%%%%%%%%%%%%%%%%%%%%%%%%%%%%%%%%%%%%%%%%%%%%%%%%%%%%%%%%%%%%%%%%%%%%%
\end{Corollary}
%%%%%%%%%%%%%%%%%%%%%%%%%%%%%%%%%%%%%%%%%%%%%%%%%%%%%%%%%%%%%%%%%%%%%%

%%%%%%%%%%%%%%%%%%%%%%%%%%%%%%%%%%%%%%%%%%%%%%%%%%%%%%%%%%%%%%%%%%%%%%
%%%%%%%%%%%%%%%%%%%%%%%%%%%%%%%%%%%%%%%%%%%%%%%%%%%%%%%%%%%%%%%%%%%%%%
\section{A braided noncommutative de Finetti theorem}
%%%%%%%%%%%%%%%%%%%%%%%%%%%%%%%%%%%%%%%%%%%%%%%%%%%%%%%%%%%%%%%%%%%%%%
%%%%%%%%%%%%%%%%%%%%%%%%%%%%%%%%%%%%%%%%%%%%%%%%%%%%%%%%%%%%%%%%%%%%%%
The following general notion of conditional independence in an operator algebraic setting
can actually be seen to arise from our main result in Theorem \ref{thm:definetti-1}, 
a noncommutative version of the famous classical de Finetti theorem.
%%%%%%%%%%%%%%%%%%%%%%%%%%%%%%%%%%%%%%%%%%%%%%%%%%%%%%%%%%%%%%%%%%%%%%
\begin{Definition} \normalfont
%%%%%%%%%%%%%%%%%%%%%%%%%%%%%%%%%%%%%%%%%%%%%%%%%%%%%%%%%%%%%%%%%%%%%%
Given the probability space $(\cA,\trace)$ let $\cN$ be a von Neumann subalgebra of $\cA$. The sequence $(x_n)_{n \ge 0} \subset \cA$ is \emph{fully $\cN$-independent}
if 
\[
E_{\cN}(xy) = E_{\cN}(x) E_{\cN}(y)
\]
for $x \in \vN\set{\cN, x_i}{i \in I}$ and $x \in \vN\set{\cN, x_j}{j \in J}$ whenever $I$ and $J$ are disjoint subsets of $\Nset_0$. 
%%%%%%%%%%%%%%%%%%%%%%%%%%%%%%%%%%%%%%%%%%%%%%%%%%%%%%%%%%%%%%%%%%%%%%
\end{Definition}
%%%%%%%%%%%%%%%%%%%%%%%%%%%%%%%%%%%%%%%%%%%%%%%%%%%%%%%%%%%%%%%%%%%%%%

%%%%%%%%%%%%%%%%%%%%%%%%%%%%%%%%%%%%%%%%%%%%%%%%%%%%%%%%%%%%%%%%%%%%%%
\begin{Theorem}\label{thm:definetti-1}
%%%%%%%%%%%%%%%%%%%%%%%%%%%%%%%%%%%%%%%%%%%%%%%%%%%%%%%%%%%%%%%%%%%%%%
Let $(\cA,\trace)$ be a tracial probability space and  suppose the
sequence $(x_n)_{n \ge 0} \subset \cA$ generates the von Neumann algebra $\cA$.
Consider the following statements:  
\begin{enumerate}
\item[(a)] 
$(x_n)_{n \ge 0}$ is exchangeable;
\item[(b)]
$(x_n)_{n \ge 0}$ is braidable;
\item[(c)]
$(x_n)_{n \ge 0}$ is spreadable;
\item[(d)]
$(x_n)_{n \ge 0}$ is stationary and fully $\cA^\cT$-independent;
\item[(e)]
$(x_n)_{n \ge 0}$ is identically distributed and fully $\cA^\cT$-independent.
\end{enumerate}
Then it holds (a) $\Rightarrow$ (b) $\Rightarrow$ (c) $\Rightarrow$ (d) $\Rightarrow$ (e). 
%%%%%%%%%%%%%%%%%%%%%%%%%%%%%%%%%%%%%%%%%%%%%%%%%%%%%%%%%%%%%%%%%%%%%%
\end{Theorem}
%%%%%%%%%%%%%%%%%%%%%%%%%%%%%%%%%%%%%%%%%%%%%%%%%%%%%%%%%%%%%%%%%%%%%%
Since we have already seen the implications
\begin{center}
exchangeable $\Rightarrow$ braidable $\Rightarrow$ spreadable $\Rightarrow$ 
stationary $\Rightarrow$ identically distributed 
\end{center}
(compare Lemma \ref{lem:ds} and Theorem \ref{thm:braidable}), the 
main difficulty is to show that spreadability implies full $\cA^\cT$-independence.
We refer to \cite{Koes10a} and give only a few hints and an example for one of the basic ideas in the proof. 

We know already from Corollary \ref{cor:tail-spread} that spreadability implies
spreadability with respect to the conditional expectation $E_\cT$:
\begin{align*}\tag{\text{*}} \label{eq:*}
(x_0, x_1, x_2, \ldots) \stackrel{E_\cT}{=}  (x_{n_0}, x_{n_1}, x_{n_2}, \ldots)
\end{align*}
for any increasing subsequence $(n_0, n_1, n_2, \ldots)$ of $(0,1,2,\ldots)$.
Let us study this property in an example to see how it produces factorization 
properties. We put $E:= E_\cT$ for notational convenience. 

\begin{Example}\normalfont
Consider the two monomials 
\begin{align*}
&&x &= x_1 x_4^3 x_5 x_4 &&\text{with index set $I=\{1,4,5\}$},&& \\
&&y &= x_8 x_6^3 x_8     &&\text{with index set $J = \{6,8\}$}.&& 
\end{align*}
Notice that we have
$I < J$ (in the pointwise sense). We infer from spreadability \eqref{eq:*} that
\begin{eqnarray*}
E(xy) 
&=& E\Big( (x_1 x_4^3 x_5 x_4) (x_8 x_6^3 x_8)\Big) \\
&=& E\Big( (x_1 x_4^3 x_5 x_4) (x_{8+k} x_{6+k}^3 x_{8+k})\Big)   \\      
&=& E\Big( (x_1 x_4^3 x_5 x_4) \alpha^k(x_{8} x_{6}^3 x_{8})\Big) 
\end{eqnarray*}
for any $k > 0$. For the last equation we have used that spreadability implies
stationarity. Now we pass to the mean ergodic average:
\[
E(xy) = E\Big( (x_1 x_4^3 x_5 x_4) \frac{1}{n}\sum_{k=0}^{n-1}\alpha^k(x_{8} x_{6}^3 x_{8})\Big) 
\]
and conclude with the von Neumann mean ergodic theorem that
\[
E_{\cA^\alpha}(y) = \lim_n\frac{1}{n}\sum_{k=0}^{n-1}\alpha^k(y)
\]
(in the strong operator topology). Next we use our fixed point characterization $\cA^\alpha = \cA^\cT$ 
from Theorem \ref{thm:fix} and the module property of conditional expectations to obtain the factorization 
\[
E(xy) = E(x E(y)) = E(x) E(y). 
\]
\end{Example}  
This example generalizes of course to any two monomials 
$x= x_{\ii(1)}^{\epsilon_1} \cdots x_{\ii(r)}^{\epsilon_r}$ and 
$y= x_{\jj(1)}^{\epsilon_1^\prime} \cdots x_{\jj(s)}^{\epsilon_s^\prime}$
as long as their index sets $I= \operatorname{Range} \ii$ and $J = \operatorname{Range}\jj$ satisfy $I < J$ or $I > J$.  

The general case of disjoint interlacing sets $I$ and $J$ is much more challenging to prove because
the mean ergodic argument above fails in such a situation. A proof of this case needs an 
order-preserving refined version of the von Neumann mean ergodic theorem. Its formulation involves 
the partial shifts $\alpha_N$ which characterize spreadability according to Proposition
\ref{prop:ds}. See \cite[Section 8]{Koes10a}.

%%%%%%%%%%%%%%%%%%%%%%%%  
 %%%%%%%%%%%%%%%%%%%%%%%%%%%%%%
\section{Characters}
The investigation of representations of the infinite braid group $\Bset_\infty$ is a vast field with many deep results and many open questions. The results of the previous sections
open up a new operator algebraic approach. As an illustration of our approach we concentrate in the following on the theory of characters. 

We start from the setting of Example \ref{ex:group} with $G= \Bset_\infty$ and $\chi$ a character of $\Bset_\infty$. The noncommutative probability space is $(\cA,\trace)$
where $\cA = \vN\set{\pi(g)}{g\in G}$ is generated by the unitary representation
$\pi$ associated to $\chi$. Further we consider the adjoint representation
\begin{eqnarray*}
\rho\colon \Bset_\infty &\rightarrow& \Aut{\cA,\trace},  \\
\tau &\mapsto& \Ad\pi(\tau) = \pi(\tau) \bullet \pi(\tau)^*.
\end{eqnarray*}

Can we find a braidable sequence here? A natural first idea is to try the constructive procedure described in Section \ref{section:ds} on the first Artin generator $\sigma_1$. However we notice that, apart from very special cases, localization \eqref{eq:L} (see Proposition \ref{prop:ds}) does not work because $u_1 := \pi(\sigma_1)$  does not commute with $u_i := \pi(\sigma_i)$ for all $i \ge 2$. But now a moment's reflection makes it clear how to overcome this difficulty: Instead of $\rho$ we have to consider a representation obtained from it by shifting the Artin generators. We also include an inversion in its definition which is purely conventional but which helps us later to obtain sequences which are already known from other points of view. In short, we consider the representation $\rho_1$ determined on the Artin generators $\sigma_i,\,i\in\Nset$, by
\[
\rho_1(\sigma_i) := \rho(\sigma^{-1}_{i+1})\,.
\]
In other words, for all $x \in \cA$ we have
\[
\rho_1(\sigma_i) x = u^{-1}_{i+1} \,x\, u_{i+1}\,.
\]
(Note that, by definition, $\rho_1$ is a representation and hence, apart from special cases, it is not equal to the composition of $\rho$ and inversion which is an anti-representation.)

Now it follows from the commutation relations \eqref{eq:B2} between the Artin generators that localization \eqref{eq:L} works for 
$u_1$ and the representation $\rho_1$, and it yields a braidable sequence $\big(v_i := \pi(\gamma_i) \big)_{i \in \Nset}$ where
\begin{eqnarray*}
\gamma_1 &:=& \sigma_1, \\
\gamma_2 &:=& \sigma^{-1}_2 {\sigma_1} \sigma_2 =
\sigma_1 {\sigma_2} \sigma_1^{-1},  \\
&\ldots& \\    
\gamma_i &:=& 
(\sigma^{-1}_i \sigma_{i-1}^{-1}\cdots \sigma^{-1}_2){\sigma_1} (\sigma_2 \cdots \sigma_{i-1} \sigma_i)
=
(\sigma_1 \sigma_2 \cdots \sigma_{i-1}){\sigma_i} (\sigma_{i-1}^{-1} \cdots \sigma_{2}^{-1} \sigma_1^{-1}).
\end{eqnarray*}

%%%%%%%%%%%%%%%%%%%%%%%%%%%%%%%%%%%%%%%%%%%%%%%%%%%%%%%%%%%%%%%%%%%%%%%%%%%
%%%%%%%%%%%%%%%%%   FIGURE 0b -- BEGIN  %%%%%%%%%%%%%%%%%%%%%%%%%%%%%%%%%%%%
%%%%%%%%%%%%%%%%%%%%%%%%%%%%%%%%%%%%%%%%%%%%%%%%%%%%%%%%%%%%%%%%%%%%%%%%%%%
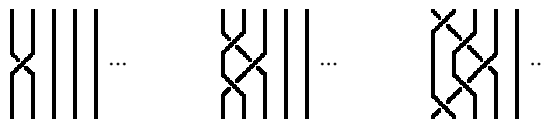
\begin{figure}[h]
\setlength{\unitlength}{0.14mm}
\begin{picture}(500,120)%(0,0) optional shift of picture
%\put(0,0){\line(0,1){100}} %reference coordinate system
%\put(0,0){\line(1,0){480}} %reference coordinate system 
%%%%%%%%%%%%%%%%%%%%%%%%%%%%%%%%%%%%%%%%%%%%%%%%%%%%%%%%%%%%
%%%%%%%%%%% Definition of partial pictures - BEGIN %%%%%%%%%
%%%%%%%%%%%%%%%%%%%%%%%%%%%%%%%%%%%%%%%%%%%%%%%%%%%%%%%%%%%%
\savebox{\artin}(20,20)[1]{\masterartin} 
\savebox{\artininv}(20,20)[1]{\masterartininv} 
\savebox{\strandr}(20,20)[1]{\masterstrandr} 
\savebox{\strandl}(20,20)[1]{\masterstrandl} 
\savebox{\horizontaldots}(20,20)[1]{\masterhorizontaldots}
%%%%%%%%%%%%%%%%%%%%%%%%%%%%%%%%%%%%%%%%%%%%%%%%%%%%%%%%%%%%
%%%%%%%%%%% Definition of partial pictures - END   %%%%%%%%%
%%%%%%%%%%%%%%%%%%%%%%%%%%%%%%%%%%%%%%%%%%%%%%%%%%%%%%%%%%%%
%%%%%%%%% the grid is 20 by 20 for partial pictures %%%%%%%%
%%%%%%%%%%%%%%%%%%%%%%%%%%%%%%%%%%%%%%%%%%%%%%%%%%%%%%%%%%%%
%%%%%%%%%%%%%%%%%%%%%%%%%%%%%%%%
\put(00,80){\usebox{\strandl}}  
\put(00,80){\usebox{\strandr}}  
\put(20,80){\usebox{\strandr}}
\put(40,80){\usebox{\strandr}}
\put(60,80){\usebox{\strandr}}
%%%%%%%%%%%%%%%%%%%%%%%%%%%%%%%%
\put(00,60){\usebox{\strandl}}
\put(00,60){\usebox{\strandr}}  
\put(20,60){\usebox{\strandr}}
\put(40,60){\usebox{\strandr}}
\put(60,60){\usebox{\strandr}}
%%%%%%%%%%%%%%%%%%%%%%%%%%%%%%%%
\put(00,40){\usebox{\artin}}  
\put(20,40){\usebox{\strandr}}
\put(40,40){\usebox{\strandr}}
\put(60,40){\usebox{\strandr}}
\put(80,40){\usebox{\horizontaldots}}
%%%%%%%%%%%%%%%%%%%%%%%%%%%%%%%%
\put(00,20){\usebox{\strandl}}
\put(00,20){\usebox{\strandr}}  
\put(20,20){\usebox{\strandr}}
\put(40,20){\usebox{\strandr}}
\put(60,20){\usebox{\strandr}}
%%%%%%%%%%%%%%%%%%%%%%%%%%%%%%%%
\put(00,0){\usebox{\strandl}}
\put(00,0){\usebox{\strandr}}  
\put(20,0){\usebox{\strandr}}
\put(40,0){\usebox{\strandr}}
\put(60,0){\usebox{\strandr}}
%%%%%%%%%%%%%%%%%%%%%%%%%%%%%%%%
%%%%%%%%%%%%%%%%%%%%%%%%%%%%%%%%
\put(200,80){\usebox{\strandl}}  
\put(200,80){\usebox{\strandr}}  
\put(220,80){\usebox{\strandr}}
\put(240,80){\usebox{\strandr}}
\put(260,80){\usebox{\strandr}}
%%%%%%%%%%%%%%%%%%%%%%%%%%%%%%%%
\put(200,60){\usebox{\artin}}  
\put(220,60){\usebox{\strandr}}
\put(240,60){\usebox{\strandr}}
\put(260,60){\usebox{\strandr}}
%%%%%%%%%%%%%%%%%%%%%%%%%%%%%%%%
\put(200,40){\usebox{\strandl}}  
\put(220,40){\usebox{\artin}}
\put(240,40){\usebox{\strandr}}
\put(260,40){\usebox{\strandr}}
\put(280,40){\usebox{\horizontaldots}}
%%%%%%%%%%%%%%%%%%%%%%%%%%%%%%%%
\put(200,20){\usebox{\artininv}}  
\put(220,20){\usebox{\strandr}}
\put(240,20){\usebox{\strandr}}
\put(260,20){\usebox{\strandr}}
%%%%%%%%%%%%%%%%%%%%%%%%%%%%%%%%
\put(200,0){\usebox{\strandl}}
\put(200,0){\usebox{\strandr}}  
\put(220,0){\usebox{\strandr}}
\put(240,0){\usebox{\strandr}}
\put(260,0){\usebox{\strandr}}
%%%%%%%%%%%%%%%%%%%%%%%%%%%%%%%%
%%%%%%%%%%%%%%%%%%%%%%%%%%%%%%%%
\put(400,80){\usebox{\artin}}  
\put(420,80){\usebox{\strandr}}
\put(440,80){\usebox{\strandr}}
\put(460,80){\usebox{\strandr}}
%%%%%%%%%%%%%%%%%%%%%%%%%%%%%%%%
\put(400,60){\usebox{\strandl}}  
\put(420,60){\usebox{\artin}}
\put(440,60){\usebox{\strandr}}
\put(460,60){\usebox{\strandr}}
%%%%%%%%%%%%%%%%%%%%%%%%%%%%%%%%
\put(400,40){\usebox{\strandl}}  
\put(420,40){\usebox{\strandl}}
\put(440,40){\usebox{\artin}}
\put(460,40){\usebox{\strandr}}
\put(480,40){\usebox{\horizontaldots}}
%%%%%%%%%%%%%%%%%%%%%%%%%%%%%%%%
\put(400,20){\usebox{\strandl}}  
\put(420,20){\usebox{\artininv}}
\put(440,20){\usebox{\strandr}}
\put(460,20){\usebox{\strandr}}
%%%%%%%%%%%%%%%%%%%%%%%%%%%%%%%%
\put(400,0){\usebox{\artininv}}  
\put(420,0){\usebox{\strandr}}
\put(440,0){\usebox{\strandr}}
\put(460,0){\usebox{\strandr}}
%%%%%%%%%%%%%%%%%%%%%%%%%%%%%%%%
\end{picture}
\caption{Braid diagrams of $\gamma_1$, 
$\gamma_2$ and  $\gamma_3$ (left to right)}
\label{figure:squareroot}
\end{figure}
%%%%%%%%%%%%%%%%%%%%%%%%%%%%%%%%%%%%%%%%%%%%%%%%%%%%%%%%%%%%%%%%%%%%%%%%%%%
%%%%%%%%%%%%%%%%%   FIGURE 0b -- END    %%%%%%%%%%%%%%%%%%%%%%%%%%%%%%%%%%%%
%%%%%%%%%%%%%%%%%%%%%%%%%%%%%%%%%%%%%%%%%%%%

The different formulas of the $\gamma_i$'s follow from each other by the braid relations. It is clear that, conversely, one can solve for the Artin generators, so the sequence $\big(\gamma_i\big)_{i \in \Nset}$ also generates $\Bset_\infty$. We refer to \cite{GoKo09a} for a discussion of the corresponding presentation of $\Bset_\infty$ and for connections with free probability arising from the fact that the squares of the $\gamma_i$'s generate free groups.  In view of that in \cite{GoKo09a}  the $\gamma_i$'s have been named \emph{square roots of free generators}.

We have proved that the sequence $\big(v_i = \pi(\gamma_i)\big)_{i \in \Nset}$ is braidable and hence fully independent over its tail algebra. By Theorem \ref{thm:fix} this tail algebra is equal to the fixed point algebra
\[
\cA^{\rho_1} = \cA \cap \{u_2, u_3, u_4, \ldots \}^\prime,
\]
the relative commutant of the represented Artin generators excluding $\sigma_1$. We note that $\cA^{\rho_1}$ contains
\[
\cA^\rho =\cA \cap \{u_1, u_2, u_3, \ldots \}^\prime = \cA \cap \cA^\prime = \cZ(\cA),
\]
the center of $\cA$. The center is trivial iff $\cA$ is a factor iff the character $\chi$ is extremal.

If we can identify the tail algebra $\cA^{\rho_1}$ in a more concrete way then the independence gives us a lot of structural information about the noncommutative probability space $(\cA, \trace)$ and hence about the character $\chi$. So far this program has been worked out only in some special cases. In the following we discuss these special cases, and from this discussion it should become clear to the reader what we have in mind when we talk of `structural information'.  

Recall that the \emph{group von Neumann algebra} $L(\Bset_\infty)$ is generated by the 
left-regular representation $\set{L_\sigma}{\sigma \in \Bset_\infty}$ 
of $\Bset_\infty$ on the Hilbert space $\ell^2(\Bset_\infty)$, where  
\[
L_\sigma f(\sigma^\prime) := f(\sigma^{-1}\sigma^\prime). 
\]
The canonical trace on the group von Neumann algebra is associated to a character $\chi$ given by
\[
\chi(e) = 1,\quad \chi(\tau) = 0 \;\text{ for }\; \tau \not= e
\]
and so this fits into our scheme. We refer the reader to \cite[Corollary 5.3]{GoKo09a} for a proof of the following result. 
%%%%%%%%%%%%%%%%%%%%%%%%%%%%%%%%%%%%%%%%%%%%%%%%%%%%%%%%%%%%%%%%%%%%%%%%%%%%%%%%%%%%%%%%%%%
\begin{Proposition}[\cite{GoKo09a}]
%%%%%%%%%%%%%%%%%%%%%%%%%%%%%%%%%%%%%%%%%%%%%%%%%%%%%%%%%%%%%%%%%%%%%%%%%%%%%%%%%%%%%%%%%%%
\begin{enumerate}
\item
$L(\Bset_\infty)$ is a non-hyperfinite $II_1$-factor; 
\item
$L(\langle\sigma_2,\sigma_3,\ldots\rangle) \subset L(\Bset_\infty)$ is an irreducible subfactor inclusion with infinite Jones index.
\end{enumerate}
%%%%%%%%%%%%%%%%%%%%%%%%%%%%%%%%%%%%%%%%%%%%%%%%%%%%%%%%%%%%%%%%%%%%%%%%%%%%%%%%%%%%%%%%%%%
\end{Proposition}
%%%%%%%%%%%%%%%%%%%%%%%%%%%%%%%%%%%%%%%%%%%%%%%%%%%%%%%%%%%%%%%%%%%%%%%%%%%%%%%%%%%%%%%%%%%

In particular we have $\cZ(\cA) = \cA^\rho = \Cset = \cA^{\rho_1} $. In fact, the last equality is exactly how irreducibility for subfactors is defined.
It follows that the sequence $\big( v_i \big)_{i \in \Nset}$ is fully $\Cset$-independent in this case. Written out this amounts to the following factorization property of the trace.

\begin{Corollary}[\cite{GoKo09a}]
For $I \subset \Nset$ we define
$
\cB_I := \vN(\pi(\gamma_i): i \in I).
$
Let $x\in \cB_I$ and $y \in \cB_J$ with $I \cap J = \emptyset$. Then
\begin{eqnarray*}
\trace(xy)= \trace(x) \trace(y).
\end{eqnarray*}
\end{Corollary}

%%%%%%%%%%%%%%%%%%%%%%%%%%%%%%%%%%%%%%%%%%%%%%%%
Another class of examples where a rather complete analysis has been achieved arises from the study of characters of the (infinite) symmetric group $\Sset_\infty$. Recall that a presentation of $\Sset_\infty$ can be given by adding the relations \eqref{eq:S} to Artin's presentation of the braid group $\Bset_\infty$. In other words there exists a quotient map (surjective homomorphism) from $\Bset_\infty$ to $\Sset_\infty$ which, for all $i \in \Nset$, maps the Artin generators $\sigma_i$ of $\Bset_\infty$ to the Coxeter generators $\sigma_i$ of $\Sset_\infty$. (We don't expect any confusion from this double meaning because from now on we only deal with $\Sset_\infty$.) If $\chi$ is a character of $\Sset_\infty$
then by composing it with the quotient map we obtain a character of $\Bset_\infty$. Representations of $\Sset_\infty$ can be identified with those representations of $\Bset_\infty$ which factorize through the quotient map.
This means that the definitions and results above are still available. But there are some special features of $\Sset_\infty$ which simplify our task. In fact, using our methods we can give a new fully operator algebraic proof of a famous classical result by Thoma (1964) which classifies all extremal characters of $\Sset_\infty$.

%%%%%%%%%%%%%%%%%%%%%%%%%%%%%%%%%%%%%%%%%%%%%%%%%%%%%%%%%%%%%%%%%%%%%%%%%%%%%%%
\begin{Theorem}[\cite{Thom64a}]
%%%%%%%%%%%%%%%%%%%%%%%%%%%%%%%%%%%%%%%%%%%%%%%%%%%%%%%%%%%%%%%%%%%%%%%%%%%%%%%
An extremal character of the group $\Sset_\infty$ is of the form 
\[
\chi(\sigma)= \prod_{k=2}^\infty \left(\sum_{i=1}^\infty a_i^k + (-1)^{k-1} \sum_{j=1}^\infty b_j^k\right)^{m_k(\sigma)}.
\]
Here $m_k(\sigma)$ is the number of $k$-cycles in the permutation $\sigma$
and the two sequences $(a_i)_{i=1}^\infty, (b_j)_{j=1}^\infty $ satisfy 
\begin{eqnarray*}
a_1 \ge a_2 \ge \cdots \ge 0, \qquad b_1 \ge b_2 \ge \cdots \ge 0, \qquad \sum_{i=1}^\infty a_i + \sum_{j=1}^\infty b_j \le 1.
\end{eqnarray*}
\end{Theorem}
%%%%%%%%%%%%%%%%%%%%%%%%%%%%%%%%%%%%%%%%%%%%%%%%%%%%%%%%%%%%%%%%%%%%%%%%%%%%%%%

Our new proof is fully presented in \cite{GoKo10a}. In the following let us sketch the proof of one of the main features of the classification: Thoma multiplicativity. This becomes very transparent in our setting, so much so that it may be justified to think of Thoma's theorem itself as a noncommutative de Finetti type theorem. 

The images of the square roots of free generators $\gamma_i$ (which we again call $\gamma_i$) turn out to be very special transpositions if we consider the defining action of $\Sset_\infty$ on $\{0,1,2,3,...\}$ by permutations:
\[
\gamma_i = (0,i)  \qquad (i \in \Nset).
\]
Algebraists call these transpositions \emph{star generators} of $\Sset_\infty$: pictorially they connect the elements of $\Nset$ to the `center' $0$ of the `star'.  To prove that indeed we obtain this simple form under the quotient map it is enough to track the permutations induced by the braids in Figure \ref{figure:squareroot} or in the corresponding algebraic definition.

For later use we note the following elementary but important property of star generators.
%%%%%%%%%%%%%%%%%%%%%%%%%%%%%%%%%%%%%%%%%%%%%%%%%%
\begin{Lemma}
%%%%%%%%%%%%%%%%%%%%%%%%%%%%%%%%%%%%%%%%%%%%%%%%%%%%%%%%%%%%%%%%%%%%%%%%%%%%%%%
Let $k \ge 2$. A $k$-cycle $\sigma = (n_1, n_2, n_3, \ldots, n_k) \in \Sset_\infty$ is of the form 
\[
\tau = \gamma_{n_1} \gamma_{n_2} \gamma_{n_3} \cdots \gamma_{n_{k-1}}\gamma_{n_k} \gamma_{n_1}, 
\]
where in the case $\tau(0) \neq 0$ we put
$n_1=0$  and define $\gamma_0$ to be the unit element $e$.
\end{Lemma}

%%%%%%%%%%%%%%%%%%%%%%%%%%%%%%%%%%%%%%%%%%%%%%%%%%%%%%%%%%%%%%%%%%%%%%%%%%%%%%%
The proof is left as an easy exercise. The importance is due to the following fact:

\begin{Corollary}[\cite{GoKo10a}]\label{cor:cycles}
%%%%%%%%%%%%%%%%%%%%%%%%%%%%%%%%%%%%%%%%%%%%%%%%%%%%%%%%%%%%%%%%%%%%%%%%%%%%%%%
Disjoint cycles can be expressed by disjoint sets of star generators. 
\end{Corollary}
%%%%%%%%%%%%%%%%%%%%%%%%%%%%%%%%%%%%%%%%%%%%%%%%%%

Because $(v_i = \pi(\gamma_i))_{i \in \Nset}$ arises from the constructive procedure from Section \ref{section:ds} with $v_1 = \pi(\gamma_1) = \pi(\sigma_1) = u_1$ and the representation $\rho_1$ it is clear that they are braidable. In fact they are even exchangeable in this setting because now $\rho_1$ is a representation of $\Sset_\infty$. From the de Finetti theorem, Theorem \ref{thm:definetti-1}, and Corollary \ref{cor:cycles} we immediately infer:

\begin{Proposition}[\cite{GoKo10a}] \label{prop:4.6}
Let $\tau_1$ and $\tau_2$ be disjoint cycles in $\Sset_\infty$.
Then $\vN(\pi(\tau_1))$ and $\vN(\pi(\tau_2))$ are $\cA^{\rho_1}$-independent. 
\end{Proposition}

As before we need to identify the fixed point algebra $\cA^{\rho_1}$ to make our independence results more concrete.
Let us denote by $E_0$ the conditional expectation onto $\cA^{\rho_1}$ and by $E_{-1}$ the conditional expectation onto the center $\cZ(\cA) = \cA^\rho$. Recall that $\cZ(\cA) \subset \cA^{\rho_1}$.

\begin{Lemma}[\cite{GoKo10a}]\label{lem:4.7}
Let $\gamma_{n_1}\gamma_{n_2}\gamma_{n_3} \cdots \gamma_{n_k}\gamma_{n_1}$ be a $k$-cycle (as above). Then
\[
E_{0}(v_{n_1} v_{n_2} v_{n_3} \cdots v_{n_k} v_{n_1}) 
=   
\begin{cases}
A_0^{k-1} & \text{if $n_1= 0$}\\
C_k  & \text{if $n_1 \neq 0$} 
\end{cases},
\]
where
\begin{eqnarray*}
A_0 := E_0(\pi(\gamma_1)), \qquad \qquad
C_k := E_{-1}(A_0^{k-1}) .
\end{eqnarray*}
\end{Lemma}

Note that, because the $\gamma_i$'s are transpositions, the representing operators $v_i = \pi(\gamma_i)$ are idempotent unitaries. This will be used repeatedly without further comment.
It follows that $A_0$ and the $C_k$'s are selfadjoint contractions.
We call these and similar objects \emph{limit cycles} because they are weak limits (ergodic averages) of represented cycles. A systematic study of such limit cycles is undertaken in \cite{GoKo10a}.
Here we develop just enough of this theory to be able to include a proof of the lemma.
%%%%%%%%%%%%%%%%%%%%%%%%%%
\begin{proof}
%%%%%%%%%%%%%%%%%%%%%%%%%%
We know from the discussion above that the sequence $\big(v_{i}\big)_{i\in \Nset}$ is fully $\cA^{\rho_1}$-independent.  
Thus
$
E_{0}(v_{n_1} v_{n_2} v_{n_3} \cdots v_{n_k} v_{n_1}) 
=  E_{0}\Big( v_{n_1} A_0^{k-1} v_{n_1}\Big).
$
For $n_1=0$ we are done because $\gamma_0 = e$. To prepare the argument for the case
$n_1 \neq 0$ choose an element $x_0 \in \cA^{\rho_1}$ and apply the constructive procedure with the (unshifted!) representation $\rho$ to obtain a sequence $x_0, x_1, x_2, \ldots$ 
Localization here amounts to 
the fact that $x_0$ commutes with $u_2, u_3, \ldots$, this is true because $x_0 \in \cA^{\rho_1}$. Hence the sequence
$\big(x_i\big)_{i \ge 0}$ is exchangeable and, by the noncommutative de Finetti theorem, independent over $\cZ(\cA) = \cA^\rho$.
Explicitly,  using again that $x_0 \in \cA^{\rho_1}$,
\[
x_n = \rho(\sigma_n \ldots \sigma_2 \sigma_1) x_0
= \rho(\sigma_n \ldots \sigma_2 \sigma_1 \sigma_2 \ldots \sigma_n) x_0 = \rho(\gamma_n) x_0 = v_n x_0 v_n.
\]
With these results we can now evaluate $E_{0}(v_{n_1} v_{n_2} v_{n_3} \cdots v_{n_k} v_{n_1})$ also for $n_1 \not=0$. In fact, 
\begin{eqnarray*}
E_{0}\Big( v_{n_1} A_0^{k-1} v_{n_1}\Big) 
&=& E_{-1}\Big( v_{n_1} A_0^{k-1} v_{n_1}\Big) \\
&=& E_{-1}\Big( A_0^{k-1} \Big) = C_k\,. 
\end{eqnarray*}
The first equality follows from independence over $\cZ(\cA)$, the second equality follows from the fact that $E_{-1}$, as a $\trace$-preserving conditional expectation onto the center $\cZ(\cA)$, is a center-valued trace. 
\end{proof}

Based on this analysis we obtain a lot of structural information about characters of $\Sset_\infty$ without further work. We give a summary in the following theorem.

\begin{Theorem}[\cite{GoKo10a}]
The fixed point algebra 
$
\cA^{\rho_1} = \vN(A_0, C_k \mid k \in \Nset) 
$
is commutative. Moreover the following are equivalent:
\begin{enumerate}
\item[(i)]
$\cA$ is a factor;
\item[(ii)]
The $C_k$'s are trivial;
\item[(iii)]
$\cA^{\rho_1}$ is generated by $A_0$.
\end{enumerate}
Further
$\cA^{\rho_1} = \Cset$ iff the (subfactor) inclusion $\vN(\Sset_{2,\infty}) \subset \vN(\Sset_{\infty})$ is irreducible. 
\end{Theorem}

\begin{proof}
Any permutation is a product of disjoint cycles. Combining Proposition \ref{prop:4.6} and Lemma \ref{lem:4.7} we find that $\cA^{\rho_1} = E_0 (\cA)$ is generated by $A_0$ and by the $C_k$'s. Because the $C_k$'s are in the center it is clear that $\cA^{\rho_1}$ is commutative. The other statements are obtained by considering special cases.
\end{proof}
%%%%%%%%%%%%%%%%%%%%%%%%%%%%%%%%%%%%%%%%%%%%%%%%%%

%%%%%%%%%%%%%%%%%%%%%%%%%%%%%%%%%%%%%%%%%%%%%%%%%%%%%%%%%%%%%%%%%%%%%%%%%%%%%%%
\begin{Corollary}[Thoma Multiplicativity {\cite[Theorem 3.7]{GoKo10a}}]
%%%%%%%%%%%%%%%%%%%%%%%%%%%%%%%%%%%%%%%%%%%%%%%%%%%%%%%%%%%%%%%%%%%%%%%%%%%%%%%
Let $m_k(\tau)$ be the number of $k$-cycles in the cycle decomposition of the permutation $\tau \in \Sset_\infty$. 
Then
\[
E_{-1}\big( \tau \big)
=  \prod_{k =2}^{\infty} C_k^{m_k(\tau)} ,
\]
where $C_k = E_{-1}(A_0^{k-1})$.
\end{Corollary}

If $\cA$ is a factor, then $E_{-1}$ can be replaced by the tracial state $\trace$ and we obtain
\[
\trace(\tau) =  \prod_{k =2}^{\infty} 
\big[\trace(A_0^{k-1})\big]^{m_k(\tau)}.
\]
%%%%%%%%%%%%%%%%%%%%%%%%%%%%%%%%%%%%%%%%%%%%%%%%%%%%%%%%%%%%%%%%%%%%%%%%%%%%%%%

%%%%%%%%%%%%%%%%%%%%%%%%%%%%%%%%%%%%%%%%%%%%%%%%%%%%%%%%%%%%%%
Thoma's theorem \cite{Thom64a} deals with extremal characters and hence with the factorial case. One of the most remarkable features of Thoma's classification is the fact that all extremal characters are multiplicative with respect to the disjoint cycle decomposition. We have reproduced this fact (Thoma multiplicativity) above. In fact, we have already achieved more. Our formula suggests that the explicit form in Thoma's theorem is now achievable with the help of a spectral analysis for the selfadjoint contraction $A_0$. Noncommutative independence continues to be a guide for this part of the proof. We refer to \cite{GoKo10a} for the details.

%%%%%%%%%%%%%%%%%%%%%%%%%%%%%%%%%%%%%%%%%%%%%%%%%%%%%%%%%%%%%%%%%%%%%%%%%%%%%%%
%%%%%%%%%%%%%%%%%%%%%%%%%%%%%%%%%%%%%%%%%%%%%%%%%%%%%%%%%%%%%%%%%%%%%%%
%%%%%%%%%%%%%%%%%%%%  B I B L I O G R A P H Y  %%%%%%%%%%%%%%%%%%%%%%%%
\bibliographystyle{alpha}                 %%  extract with BibTeX
\label{section:bibliography}
%\bibliography{lib-symmetric}                   %%  BibTeX-File

%%%%%%%%%%%%%%%%%%%%%%%%%%%%%%%%%%%%%%%%%%%%%%%%%%%%%%%%%%%%%%%%%%%%%%%
\end{document}